\documentclass[review]{elsarticle}

\usepackage{lineno,hyperref}
\usepackage[left=2cm,top=2.5cm,right=2cm,bottom=2.5cm]{geometry}
\usepackage{graphicx,array,delarray}
\usepackage{subfigure}
\usepackage{latexsym}
\usepackage{amscd}
\usepackage{amsfonts, amsmath, amssymb}
\usepackage[utf8]{inputenc}
\usepackage{verbatim}
\usepackage{mathrsfs}
\usepackage{fancyhdr}
\usepackage{float}
\usepackage{palatino}
\usepackage[english]{babel}
\usepackage[usenames]{color}

\modulolinenumbers[5]
\newcommand{\RR}{\mathbb{R}}

\newcommand{\R}{\mathbf{R}}

\newtheorem{theorem}{\textbf{Theorem}}

\newtheorem{corollary}[theorem]{\textbf{Corollry}}

\newtheorem{example}{\textbf{Example}}
\newtheorem{definition}[theorem]{\textbf{Definition}}

\newtheorem{lemma}[theorem]{\textbf{Lemma}}

\newtheorem{proposition}[theorem]{\textbf{Proposition}}

\newenvironment{proof}[1][Proof]{\noindent\textbf{#1.} }{\ \rule{0.5em}{0.5em}}

\newcommand{\tto}{\longrightarrow}
\DeclareMathOperator{\N}{\mathbb{N}}
\DeclareMathOperator{\Diam}{Diam}  
\newcommand{\mitext}[1]{\quad\text{#1}\quad}

\def\conv{{\rm conv}}

\begin{document}

\begin{frontmatter}
\title{Some results of fixed point of non-expansive mappings on asymmetric spaces}


\author[mymainaddress0]{L. Ben\'itez-Babilonia\corref{mycorrespondingauthor}}
\cortext[mycorrespondingauthor]{Corresponding author}
\ead{luisbenitez@correo.unicordoba.edu.co}

\author[mymainaddress]{R. Felipe}
\author[mymainaddress0]{L. Rubio}


\address[mymainaddress0]{Departamento de Matem\'{a}ticas y Estad\'{\i}stica, 
Universidad de C\'{o}rdoba,
Cra $6$, No. $77-305$ 
Monter\'{\i}a, Colombia.}
\address[mymainaddress]{Centro de Investigaci\'on en Matem\'aticas A.C., 
Callej\'on Jalisco s/n Mineral de Valenciana,
Guanajuato, Gto., M\'exico.}

\begin{abstract}
{Some fixed point results of classical theory, such as Banach's Fixed Point Theorem, have been previously extended by other authors to asymmetric spaces in recent years. The aim of this paper is to extend to asymmetric spaces some others fixed point results for contractions, shrinkage maps and non-expansive maps. In fact, a version of Edelstein type theorem (Theorem~\ref{res1}), Schauder type theorem (Theorem~\ref{res2}), and Kirk type theorem (Theorem~\ref{Kirk}) are stated and proved in this new context. In order to do that, classical definitions and results were adapted to this new context. Also, the normal structure in the asymmetric case was considered.}
\end{abstract}


\end{frontmatter}

\section{Introduction}

Asymmetric spaces were first introduced by Wilson in \cite{Wil} as quasi-symmetric spaces. An asymmetric space is a generalization of a metric space where the distance does not satisfy the axiom of symmetry. In this case there are two topologies defined on the same space, the forward topology  $\tau^{f}$ and the backward topology $\tau^{b}$, as can be found in \cite{Kel}, then, for notions such as convergence, completeness and compactness two versions are considered, namely forwards and backwards. In \cite{ColZim},  these concepts were studied in the asymmetric context, and complemented in \cite{Cob1}, where are given the basic results on asymmetric normed spaces. In some cases, these spaces represent a specific topic of interest, for example, within the geometric theory of groups, very interesting results have been found related to exterior spaces, as can be seen in \cite{AlKfBest}.

The study of fixed points for different types of applications, in different types of spaces, with certain geometric and topological properties in classical theory, have been extended to asymmetric spaces. For instance, in \cite{KhoMoMou}, is introduced the notion of a forward and backward contraction to prove the Banach Contraction Principle in asymmetric spaces, and in \cite{Cob2}, a generalization of this theorem is given, which is known as Caristi-Kirk fixed point theorem; also, in \cite{SeMaWar}, are presented some fixed-point results for applications of the $\chi F$ type contractions and applications to fractal theory.

This article is organized as follows. Since our aim is to give a precise and self-contained proofs of fixed point theorems in asymmetric spaces, in Sections~\ref{Sec2} and \ref{Sec3}, a brief exposition of some topological concepts that have been extended to asymmetric spaces and asymmetric normed spaces is included, which are well known and have been widely studied in the references \cite{Alegre}, \cite{Cob1} and \cite{Garcia4}.

In Section \ref{Sec4}, the definition of  normal structure is extended to the context of asymmetric normed spaces, taking into consideration the approach appeared in \cite{GoeKir} about the classic theory of normal structure due to M. Brodski and D. Milman. The  normal structure is a geometric property of convex subsets of a usual Banach space. From this notion some interesting results were obtained in the fixed-point theory of non-expansive maps on Banach spaces. However, as far as we know, in the asymmetric case this type of structure has not received any attention.
We consider that our work is a first step in order to fill this gap.

Finally, in Section \ref{Sec5} the main results are stated and proved. The Banach Contraction Principle, in asymmetric spaces \cite{KhoMoMou} is taken as a starting point, to give another versions of it, which are Theorem~\ref{PCBA2} and Corollary~\ref{coro1}, respectively. Asymmetric versions of classical theorems of fixed point theory are also presented, such as: Edelstein type Theorem (Theorem~\ref{res1}), Schauder type Theorem (Theorem~\ref{res2}), Theorem~\ref{Teo3} and Theorem~\ref{Kirk}.

\section{Asymmetric spaces} \label{Sec2}

In this section, we focus on the basic topological theory derived from the definition of asymmetric distance. See \cite{Men} for a more complete study of this subject.


\begin{definition}[\cite{Cob1}]\label{def1}
Let $X$ be a non-empty set. A function 
is called an \textbf{asymmetric distance} on $X$ if satisfies the following conditions:
\begin{enumerate}
\item[AD1.] if $x,y\in X$, then $p(x,y)\geq 0$,
\item[AD2.] $p(x,y)=p(y,x)=0$ if and only if $x=y$,
\item[AD3.] if $x,y,z\in X$, then $p(x,z)\leq p(x,y)+p(y,z)$.
\end{enumerate}
In such a case, the pair $(X,p)$ is called an {\bf asymmetric space}.
\end{definition}

Note that, in these spaces, the symmetric condition is not required. That is, $p(x,y)$ and $p(y,x)$ are not necessarily equal. Then, two topologies $\tau^{f}$ and $\tau^{b}$ can be defined in $X$, which are generated by the bases \mbox{$\mathscr{B}_1:=\{B^{f}(x,r):x\in X\mitext{and} r>0\}$} and \mbox{$\mathscr{B}_2:=\{B^{b}(x,r):x\in X\mitext{and} r>0\}$}, respectively (see \cite{Kel}). Here, we have used the following notation:
\begin{align*}
    B^{f}(x,r)&=\{u\in X: p(x,u)<r\} \hfill \quad\text {{for the {\bf open $f-$ball}}},\\
    B^{b}(x,r)&=\{v\in X: p(v,x)<r\} \hfill \quad\text {{for the {\bf open $b-$ball}}}.
\end{align*}
Thus, a set $A\subseteq X$ is {\bf $f-$open} ({\bf $b-$open}) in $X$, if, for each $x\in A$, there is $r>0$ such that $B^{f}(x,r)\subset A$ ($B^{b}(x,r)\subset A$, respectively). See \cite{Cob1} for more details.

If $(X,p_1)$ and $(Y,p_2)$ are two asymmetric spaces and $g:X\tto Y$ is a function, then four different notions of continuity can be defined, according to the considered topologies. However, such continuities are characterized by the sequential continuity, as follows: Let $(X,p)$ be an asymmetric space. The sequence $\{x_n\}\subset X$ is {\bf $f-$convergent} (respectively {\bf $b-$convergent}) to $x_0\in X$, denoted by $x_n\overset{f}{\tto} x_0$ (respectively $x_n\overset{b}{\tto} x_0$), if \mbox{$p(x_0,x_n)\tto0$} (respectively $p(x_n,x_0)\tto0$).

\begin{example}
Consider the function $d:\R\times\R\rightarrow[0,\infty)$ defined by means of
\begin{equation}
d(x,y)=
\begin{cases}
    y-x,  & \quad\text{if} \quad y>x;   \\
    0,  &\quad\text{if} \quad  y\leq x.
  \end{cases}
\end{equation}
Then, $d$ is an asymmetric distance and $(\R,d)$ is an asymmetric space. In this example, if $x<y$, then $d(x,y)>0$ and $d(y,x)=0$. More examples can be found in \cite{ColZim}, \cite{KhoMoMou} and \cite{SeMaWar}.
\end{example}

There are many notion of Cauchy sequence that are related but not equivalent in the asymmetric case, as can be seen in \cite{Cob1}. We present the most convenient to our work.

\begin{definition}[\cite{ColZim}]\label{cauchy}
Let $(X,p)$ be an asymmetric space. A sequence $\lbrace x_{n}\rbrace$ in $X$ is {\bf $f-$Cauchy} (resp.\,{\bf $b-$Cauchy}), if for all $\epsilon>0$ there exists $ N\in \N$ such that for \mbox{$m\geq n\geq N$}, $p(x_{n},x_{m})<\epsilon$ (resp. $p(x_{m},x_{n})<\epsilon$) holds.
 \end{definition}
The order of the subscripts $n$ and $m$ must be considered in asymmetric spaces. Kelly, Collins and Zimmer present some examples of $f-$convergent sequences which are not $f-$Cauchy. See \cite{Kel}, Example 5.8 and \cite{ColZim}, Example 3.6. However, other properties are preserved in asymmetric spaces, as it is showed in the following proposition.

\begin{proposition}[\cite{ColZim}]\label{k12}
Let $(X,p)$ be an asymmetric space and $\left\lbrace x_{n}\right\rbrace$ be a $b-$Cauchy sequence on $X$. If $\left\lbrace x_{n}\right\rbrace$ has an $f-$convergent subsequence, then $\left\lbrace x_{n}\right\rbrace$  $f-$convergent.
\end{proposition}
\begin{proof}
Let $\left\lbrace x_{n_{k}}\right\rbrace\subset \left\lbrace x_{n}\right\rbrace$ be a subsequence such that $\displaystyle{x_{n_{k}} \stackrel{f}{\rightarrow} x_{0}}$. Then, for all $\epsilon>0$, there exists $N_{1}\in \N$ such that
\begin{equation}\label{fss2}
   p(x_{0},x_{n_{k}})<\frac{\epsilon}{2}\quad\text{for all}\quad k\geq N_{1}.
\end{equation}
On the other hand, for all $\epsilon>0$, there exists $N_{2}\in\N$ such that
 \begin{equation}\label{bsC1}
    p(x_{n_{k}},x_{n})<\frac{\epsilon}{2}\quad\text{for all}\quad n_{k}\geq n\geq N_{2}.
 \end{equation}
Let $N_{0}=\max\left\lbrace N_{1},N_{2}\right\rbrace$, $k\geq N_{0}$ and $n_{k}\geq k$. Thus, by (\ref{fss2}) and (\ref{bsC1}),
\begin{equation*}
    p(x_{0},x_{n})\leq p(x_{0},x_{n_{k}})+p(x_{n_{k}},x_{n})<\epsilon.
\end{equation*}
 \end{proof}

It is worth pointing out that compactness and sequential compactness do not always coincide in asymmetric spaces.

\begin{definition}[\cite{ColZim}]\label{comp}
Let $(X,p)$ be an asymmetric space. A subset $K\subset X$ is {\bf $f-$compact} (resp. {\bf $b-$compact}) if every open cover of $K$ in $\tau^{f}$ (resp. $\tau^{b}$) has a finite subcover. We will say that $K$ is {\bf relatively $f-$compact} (resp. {\bf relatively $b-$compact}) if $cl^{f}(K)$ is $f-$compact (resp. $cl^{b}(K)$ is $b-$compact), where $cl^{f}$ denotes the closure in $\tau^{f}$ (resp.  $cl^{b}$ denotes the closure in $\tau^{b}$). Also, $K\subset X$ is {\bf sequentially $f-$compact} (resp. {\bf sequentially $b-$compact}) if every sequence has an $f-$convergent subsequence (resp. $b-$convergent subsequence) with limit at $K$.
\end{definition}

Note that different types of Cauchy sequences provide diverse types of completeness in asymmetric spaces, each with its advantages and disadvantages. Thus, the next definition is related to Definition~\ref{cauchy}.
\begin{definition}[\cite{ColZim}]
An asymmetric space $(X,p)$ is {\bf $f-$complete} (resp.  {\bf $b-$complete}), if every $f-$Cauchy sequence is $f-$convergent (resp. $b-$Cauchy sequence is $b-$convergent) on $X$. We say that a subset $K\subset X$ is {\bf totally $f-$bounded} (resp.  {\bf totally $b-$bounded)}, if for every $\epsilon>0$, $K$ can be covered by a finite number of $f-$open balls (resp. $b-$open balls) of radius $\epsilon$.
\end{definition}

On the other hand, if $(Y,d)$ is a metric space, then clearly $(Y,\tau^{f})=(Y,\tau^{b})$. Thus, considering an asymmetric space $(X,p)$, we have two types of continuities for a function $g:X\tto Y$: $f-$continuity and $b-$continuity. Then, in the particular case when $(Y,d_{Y})=(\RR,\vert \cdot\vert)$, we have the following results\,:

\begin{proposition}\label{prop4}
Let $(X,p)$ be an $f-$compact asymmetric space and consider $\RR$ with the usual metric. If $f:X\to\RR$ is $f-$continuous on $X$, then $f(X)$ is compact on $\RR$.
\end{proposition}

Since the proof of the preceding proposition is, except for obvious modifications, identical to that in the symmetric framework, it will be omitted here. \\

The next proposition is an adaptation to asymmetric spaces of the Weierstrass Theorem.
In this case, the result is addressed using a function $f$ that maps a compact asymmetric space to a metric space.

\begin{proposition}\label{prop6}
Let $(X,p)$ be an $f-$compact asymmetric space and let us consider $\RR$ provided with the usual metric. Suppose that $f:X\to\RR$ is $f-$continuous on $X$ and
\begin{equation}
   M=\sup_{x \in X} f(x), \quad m=\inf_{x\in X}f(x).
\end{equation}
Then, there exist $x_1,x_2\in X$ such that $f(x_1)=M$ and $f(x_2)=m$.
\end{proposition}
\begin{proof}
It follows immediately from Proposition~\ref{prop4}.
\end{proof}

\section{Asymmetric Functional Analysis }\label{Sec3}

In the present section, we introduce the asymmetric normed spaces, which we denote by $(X,\Vert \cdot \vert)$. Most of the results and the theory presented here are based on the works \cite{Alegre}, \cite{Cob1} and \cite{Garcia4}.

\begin{definition}[\cite{Cob1}]\label{NA}
Let $X$ be a real vector space. We say that $ \Vert \cdot \vert: X \to \RR$ is {\bf an asymmetric norm} on $ X $ if it satisfies the following properties:
\begin{enumerate}
\item[AN1.] For each $ x \in X,  \Vert x \vert\geq 0$.
\item[AN2.] For each $ x \in X$, $\Vert x \vert=\Vert -x \vert =0$ if and only if, $x = 0$.
\item[AN3.] For each $\lambda\geq 0 $ and each $ x \in X$, $\Vert \lambda x \vert = \lambda \Vert x \vert$.
\item[AN4.] For each $x, y \in X$, $\Vert x + y \vert\leq \Vert x \vert + \Vert y \vert$.
\end{enumerate}
In this case, the pair $(X,\Vert\cdot\vert)$ is called {\bf an asymmetric normed space}. 
\end{definition}
\begin{example}\label{NAE!}
The function $\Vert\cdot \vert:\RR^{2}\to\RR$ defined by $\Vert(x,y)\vert=\max\left\lbrace 0,y-x,y+x\right\rbrace $ is an asymmetric norm in $\RR^2$. We refer the reader to \cite{Cob1} for more details.
\end{example}

\begin{example}[\cite{Cob1}]
The function $\Vert \cdot\vert_{u}:\RR\to \RR^{+}$ defined by
\begin{equation}\label{upper}
\Vert x\vert_{u}=x\vee 0=\max\left\lbrace x,0\right\rbrace
\end{equation}
is an asymmetric norm.
\end{example}
An asymmetric norm induces two asymmetric distances, namely \mbox{$d(x,y)=\Vert y-x\vert$} and $\hat{d}(x,y)=\Vert x-y\vert$.
\begin{definition}[\cite{Cob1}]\label{BAsim}
The space $(X,\Vert \cdot\vert)$ is said to be {\bf $f-$Banach} if it is $f-$complete in the asymmetric distance given by $d(x,y)=\Vert y-x\vert$. Similarly, the space $(X,\Vert \cdot\vert )$ is {\bf $b-$Banach} if it is $b-$complete in the asymmetric distance given by $\hat{d}(x,y)=\Vert x-y\vert$. 
\end{definition}

Inspired by the definition of forward and backward contractions given in \cite{KhoMoMou}, we give the following definition.
\begin{definition}\label{df2}
Let $(X,d)$ be an asymmetric space. A mapping $T: X\to X$ is {\bf $f-$Lipschitz} if there exists a non-negative real number $k$ such that
\begin{equation}\label{Lip1}
        d(Tx,Ty)\leq kd(x,y),\quad \text{for all}\quad  x,y\in X.
\end{equation}
The smallest value of $k$ in (\ref{Lip1}) will be called the {\bf Lipschitz $f-$constant} of $T$ and is denoted by $k_{f}$.

Similarly, a mapping $T: X\to X$ is {\bf $b-$Lipschitz} if there exists a non-negative real number $l$ such that
\begin{equation}\label{Lip2}
d(Tx,Ty)\leq ld(y,x),\quad \text{for all}\quad  x,y\in X.
\end{equation}
The smallest value of $l$ in (\ref{Lip2}) is called the {\bf Lipschitz $b-$constant} of $T$ and denoted by $l_b$.
 \end{definition}
According to the above, we give the following definitions.
\begin{enumerate}
  \item[$\bullet$] The mapping $T$ is {\bf $f-$non-expansive} if $0\leq k_f\leq 1$; and it is a {\bf $b-$non-expansive} if $0\leq l_b\leq 1$.
  \item[$\bullet$] The mapping $T$ is an {\bf $f-$contraction} if $0\leq k_f<1$; and is a {\bf $b-$contraction} if $0\leq l_b<1$.
\end{enumerate}
Furthermore, the mapping $T:X\to X$ is called {\bf $f-$shrinkage} if 
\begin{equation}\label{fct}
    d(Tx,Ty)< d(x,y);\quad\text{for all}\quad x,y\in X,\quad\text{con}\quad x\neq y.
\end{equation}
Similarly, $T:X\to X$ is called {\bf $b-$shrinkage} if 
\begin{equation}\label{bct}
    d(Tx,Ty)< d(y,x);\quad\text{for all}\quad x,y\in X,\quad\text{con}\quad x\neq y.
\end{equation}
The details of the following example can be found in (\cite{KhoMoMou}).
\begin{example}\label{daa1}
Let $d_1:\RR\times\RR\rightarrow[0,\infty)$ be a function defined of the following form
\[
d_1(x,y)=\left\{\begin{array}{lll}
y-x, &\text{if} & y \geq x ; \\
\frac{1}{4}(x-y), &\text{if}&x>y.
\end{array}\right.
\]
Then $(\mathbb{R},d_1)$ is an asymmetric space and the mapping $T:\mathbb{R}\to\mathbb{R}$, given by $Tx=\frac{1}{2}x$, is an $f-$contraction but not a $b-$contraction, where $\RR$ is endowed with $d_1$ in both cases.
\end{example}

\subsection{Weak topology in asymmetric normed spaces}

In this new context we can introduce the weak topology. The weak topology and the topology of the asymmetric norm coincide in a wide class of finite-dimensional asymmetric normed spaces \cite{Alegre}. However, this is not necessarily true for every asymmetric space. In fact, it can be seen that the weak topology on an infinite-dimensional asymmetric normed space is strictly coarser than the topology of the asymmetric norm.

If $(X,\Vert\cdot\vert)$ is an asymmetric normed space, $X^{\ast}$ denotes the set
\begin{equation}
X^{\ast}=\left\lbrace \varphi:(X,\Vert \cdot\vert)\to (\RR,\Vert \cdot\vert_{u})\,:\, \text{$\varphi$ is linear and continuous}\right\rbrace .
\end{equation}
Then $X^{\ast}$ is the topological asymmetric dual of the asymmetric normed space  $(X,\Vert\cdot\vert)$.

\begin{definition}[\cite{Alegre}]
Let $(X,\Vert \cdot\vert)$ be an asymmetric normed space. The {\bf weak forward topology} on $X$, induced by the asymmetric norm $\Vert\cdot\vert$, is the topology generated by the base 
\begin{align*}
  \mathscr{B}_{\ast}:=\{V_{\varphi_{1}, \varphi_{2}, \ldots, \varphi_{n}}(x, \epsilon):x\in X,\quad \epsilon>0 \mitext{and} \varphi_{1}, \varphi_{2}, \ldots, \varphi_{n}\in X^{\ast}\},
\end{align*}
where \mbox{$V_{\varphi_{1}, \varphi_{2}, \ldots, \varphi_{n}}(x, \epsilon)=x+V_{\varphi_{1}, \varphi_{2}, \ldots, \varphi_{n}}(0, \epsilon)$}, and 
\begin{equation}
V_{\varphi_{1},\varphi_{2},\ldots,\varphi_{n}}(0,\epsilon)=\left\{ z\in X\, :\, \varphi_{1}(z)<\epsilon,\varphi_{2}(z)<\epsilon,\ldots,\varphi_{n}(z)<\epsilon\right\},
\end{equation}
whenever $\varphi_{1}, \varphi_{2}, \ldots, \varphi_{n}\in X^{\ast}$, $\epsilon >0$ and $n\in \N$.

The {\bf weak forward topology} induced by $\Vert\cdot\vert$ will be denoted by $\tau_{w}^{f}$ and the {\bf weak backward topology} induced by $\Vert\cdot\vert$ will be denoted by $\tau_{w}^{b}$.
\end{definition}

\begin{definition}[\cite{Alegre}]
Let $(X,\Vert \cdot\vert)$ be an asymmetric normed space. The sequence \mbox{$\left\lbrace x_{n}\right\rbrace\subset X$} is {\bf weakly $f-$convergent} to $x_{0}\in X$, if the sequence $\left\lbrace \varphi(x_{n})\right\rbrace$ is $f-$convergent to $\varphi(x_0)$ for all $\varphi\in X^{\ast}$, that is, we have
\begin{equation}
\Vert \varphi(x_{n})-\varphi(x_{0})\vert_{u}\to 0,\,\forall \varphi\in X^{\ast}.
\end{equation}
The weak $f-$convergence will be denoted by $\displaystyle{x_{n}\stackrel{f}{\rightharpoonup} x_{0}}$.
\end{definition}

Let $(X,\Vert \cdot\vert)$ be an asymmetric normed space. The sequence $\left\lbrace x_{n}\right\rbrace\subset X$ is  {\bf  strongly $f-$convergent} if it is $f-$convergent in the asymmetric norm. Below, we will call it simply $f-$convergence and write $\displaystyle{x_{n} \stackrel{f}{\rightarrow} x_{0}}$ when no confusion can arise. In an analogous way, the {\bf strong $b-$convergence} is defined.

Next, we present an asymmetric version of Mazur's theorem whose proof is analogous to the classical version which can be consulted in \cite{Yos}.
\begin{theorem}[Mazur-type theorem]\label{Mazur}
Let $X$ be an asymmetric normed space and \mbox{$\lbrace x_{n}\rbrace\subset X$} be a weakly $f-$convergent sequence to $x_{0}\in X$. Then, for all $\epsilon>0$, there exists a convex combination $\displaystyle y_{m}=\sum_{j=1}^{m}\alpha_{j}x_{n_{j}}$ such that $\Vert x_{0}-y_{m}\vert\leq \epsilon$.
\end{theorem}

\begin{proof}
Let us consider $M_1={\conv}(\lbrace 0, x_{1},x_2,x_3,\ldots\rbrace)$. Then, we have to prove that, for every $\epsilon>0$, there is $y\in M_1$, such that $\Vert x_{0}-y\vert\leq \epsilon$. In contrary case, there exists $\epsilon_{0}>0$ with $\Vert x_{0}-y\vert>\epsilon_{0}$, for all $y\in M_{1}$. Let us define the following set
\begin{align}
    M=\left\lbrace v\in X\,:\, \Vert v-u\vert\leq \frac{\epsilon_{0}}{2};\quad\text{for some}\quad u\in M_{1}\right\rbrace.
\end{align}
Note that $M\neq \emptyset$,  $M$ is a convex set, $M_{1}\subset M$ and $x_{0}\notin M$. Moreover, for any $x\in X$ there exists $\alpha>0$ such that $\alpha^{-1}x\in M$. In particular, there exists $v_0\in M$ with $v_0\neq0$ and $\beta_0\in(0,1)$ such that $x_0=\beta_0^{-1}v_0$. In this case, we can define $p=p_{M}:X\to \R$ the Minkowski functional of $M$, given by
\begin{align*}
    p(z)=\inf \lbrace t>0\,:\, z\in t M\rbrace, \mitext{for each} z\in X.
\end{align*} 
Thus, $p(v_{0})=1$ and $p(x_{0})\geq \beta_0^{-1}>1$.

Now, let us consider the functional $f_{0}:\R v_{0}\to \R$ defined by $f_{0}(tv_{0})=t$. Then, Hahn-Banach type theorem (Theorem~2.2.2,~\cite{Cob1}) implies the existence of a functional $\varphi_0:X\to \R$ which is an extension of $f_0$, such that
\begin{align}\label{mink}
    \varphi_0(x)\leq p(x), \quad \text{for every}\quad x\in X.
\end{align}
This implies
\begin{align}\label{contraEn}
\sup_{x\in M_{1}}\varphi_0(x)\leq\sup_{x\in M}\varphi_0(x)\leq\sup_{x\in M}p(x)=1<\beta_0^{-1}=\varphi_0\left(\beta_0^{-1} v_{0}\right)=\varphi_0\left(x_{0}\right).
\end{align}
On the other hand, $\displaystyle{x_{n}\stackrel{f}{\rightharpoonup}x_{0}}$, that is $\displaystyle{\phi(x_{n})\overset{f}{\tto}\phi(x_{0})}$ for all $\phi \in X^{\ast}$. In particular, ${\varphi_0(x_{n})\overset{f}{\tto}\varphi_0(x_{0})}$. But, Inequality (\ref{contraEn}) with $\lbrace x_{n}\rbrace\subset M_{1}$ imply
\begin{align}
    \sup_{n\in\N}\varphi_0(x_{n})<\varphi_0\left(x_{0}\right),
\end{align}
which is impossible. Therefore, the result is true. 
\end{proof}

In an asymmetric space $(X,d)$, a set $C\subset X$ is \textbf{weakly $f-$closed} (resp. {\bf weakly $b-$closed}) if it is closed with respect to the topology $\tau_w^{f}$ (resp. $\tau_w^{b}$). Also, it is \textbf{weakly $f-$compact}  (resp. {\bf weakly $b-$compact})  if it is compact with respect to the topology $\tau_w^{f}$ (resp. $\tau_w^{b}$).

\section{Normal structure in asymmetric normed spaces}\label{Sec4}

The concept of normal structure was introduced by Brodskii and Milman. This notion has been useful in the study of fixed points of non-expansive self-mappings on $K$, where $K$ is a weakly compact set. In this section, we introduce this notion in the context of asymmetric spaces.

For any subset $K$ of the set $X$, we define the diameter of $K$ as
\begin{equation}\label{diametro}
\Diam (K)=\sup \{\Vert v-u\vert : u,v \in K\}.
\end{equation}

A set $K\subset X$ is {\bf bounded} if $\Diam(K)<\infty$. Also, we say that a subset $K\subseteq X$ is {\bf $f$-bounded} (resp. {\bf $b$-bounded}) if it is contained in an $f$-ball (resp. in a $b$-ball). Then, any bounded non-empty subset of $X$ is $f$-bounded and $b$-bounded. Indeed, let $x\in K$, then taking into account that $\Diam(K)<\infty$, we have that for all $u\in K$
\begin{equation*}
 \Vert u-x \vert<\Diam(K)+1,
\end{equation*}
that is $K\subset B^{f}(x,r_{0})$ where $r_{0}=\Diam(K)+1$.  Similarly we show that $K\subset B^{b}(x,r_{0})$. Vice versa, if $K$ is contained both in an $f$-ball and a $b$-ball, then this is bounded. \\

Now, if $u\in X$, we define two radiuses of $K$ with respect to $u$ as
\begin{align}\label{radios}
r^{f}_{u}(K)&=\sup \{\Vert v-u\vert : v \in K\};\quad\hfill\text{forward radius of $K$,} \\
r^{b}_{u}(K) &=\sup \{\Vert u-v\vert : v \in K\};\quad\hfill\text{backward radius of $K$.}
\end{align}
Thus, a point $u\in K$ is called \textbf{forward diametral} if
\begin{equation}
 r_{u}^{f}(K)=\Diam(K),
\end{equation}
otherwise, if $r_{u}^{f}(K)<\Diam(K)$, then it is called \textbf{forward non-diametral}.

In general, a set $D\subset X$ is called \textbf{forward diametral} if all its points are forward diametral. In classical theory, the following definition refers to a geometric property of Banach spaces, which is called normal structure. Now, we will present a similar definition in the context of asymmetric normed spaces.

\begin{definition}\label{sn}
Let $X$ be an $f-$Banach space and $K$ be a convex subset of $X$. We say that $K$ has \textbf{forward normal structure}, if every convex bounded subset $H$ of $K$ with $\Diam(H)>0$ contains a forward non-diametral point of $H$, that is, there exists $u\in H$ such that
 \begin{equation}
 r^{f}_{u}(H):=\sup \{\Vert v-u\vert: v \in H\}<\Diam(H).
 \end{equation}
\end{definition}
In other words, the subsets of $X$ with forward normal structure do not contain convex bounded subsets consisting only of forward diametral points, except for those of cardinality one.

Consider $D$ a subset of $K$ and a function $T: K\to K$. The set $D$ is {\bf $T-$invariant} if $T(D)\subseteq D$.

\begin{definition}\label{def31}
Let $K$ be a non-empty, $f-$closed, convex subset of $X$ and $T:K\to K$ be a map. The set $K$ is \textbf{minimal $T-$invariant} if $T(K)\subseteq K$ and $K$ does not have a proper non-empty, $f-$closed, convex subset which is $T-$invariant.
\end{definition}

The proof of the following two results are almost analogous to the symmetric case, under appropriated changes. In \citep[page 33]{GoeKir}, the proofs of the classical versions can be found. 
\begin{proposition}\label{minimal}
Let $X$ be an $f-$Banach space and $K$ be a non-empty, convex, weakly $f-$compact subset of $X$. Then, for any map $T:K\to K$ there exists a minimal $T-$invariant set $D\subseteq K$.
\end{proposition}
\begin{proof}
Let us consider the following family of sets
\begin{equation*}
 \mathcal{P}=\left\{D\subseteq K : D\neq\emptyset, \text{ is convex,  weakly $f-$compact and $T-$invariant} \right\}.
\end{equation*}
Note that $(\mathcal{P},\subseteq)$ is a partial ordered set. Moreover, each totally ordered subfamily of $\mathcal{P}$ is lower bounded, in fact, if $\{D_{\alpha}:\alpha\in\Delta\}\subseteq\mathcal{P}$ is totally ordered, then $\displaystyle\bigcap_{\alpha\in\Delta}D_{\alpha}$ is a lower bound. Thus, Zorn's Lemma implies the existence of a minimal set $D_1\in\mathcal{P}$. Therefore, $D_1$ is a minimal $T-$invariant set.
\end{proof}

\begin{lemma}\label{mch}
If $K$ is a minimal $T-$invariant set, then \mbox{$K=\overline{\conv}(T(K))$}.
\end{lemma}
\begin{proof}
Let us set $D=\overline{\conv}(T(K))$. Note that $D$ is a non-empty, convex, $f-$closed set. Hence, the convexity and the $f-$closedness of $K$ imply \mbox{$D=\overline{conv}(T(K))\subseteq \overline{conv}(K)=K$}. That is, $D\subseteq K$. Thus,
\begin{equation*}
 T(D)\subset T(K)\subseteq \overline{conv}(T(K))=D.
\end{equation*}
This implies that $D$ is $T-$invariant. Finally, since $K$ is minimal, we obtain that $K=D=\overline{conv}(T(K))$.
\end{proof}

\section{Some fixed point results of non-expansive mappings}\label{Sec5}

We will show that $f-$bounded, $f-$closed, and convex sets $K\subset X$ in asymmetric normed spaces have the property that each non-expansive self-mapping $T:K\to K$ has a fixed point. This is obtained by assuming additional conditions on $K$ or $X$.

There are many classic fixed point results which have been extended to asymmetric spaces. A good example is the Banach contraction principle presented by \cite{KhoMoMou}.

\begin{theorem}\label{PCBA1}
Let $(X,d)$ be a complete forward space and $T: X\to X$ be a $f-$contraction. Assume that $f-$convergence implies $b-$convergence. Then, $T$ has a unique fixed point.
\end{theorem}
In fact, many of these results are built for different types of contractions, as in the case of the Caristi-Kirk theorem, which is a generalization of the Banach contraction principle presented by Cobzas in \cite{Cob2}.

The following result is a consequence of Theorem~\ref{PCBA1}.

\begin{corollary}\label{coro1}
Let $X$ be a complete forward asymmetric space and let $T: X \rightarrow X$ be a mapping such that $T^{k}$ is a $f-$contraction for some $k \in \mathbb{Z}^{+} $. Suppose that $f-$convergence implies $b-$convergence. Then, $T$ has a unique fixed point.
\end{corollary}
\begin{proof}
By Theorem~\ref{PCBA1} there exists a unique $x_{0}\in X$ such that $T^{k}x_{0}=x_{0}$, which implies $Tx_{0}=T^{k}\left(T{x_{0}}\right)$. Then, $Tx_{0}$ is also a fixed point of $T^{k}$, and so $x_{0}$ is a fixed point of $T$. In order to prove uniqueness, let us assume that $y \in X$ is also a fixed point of $T$, that is $Ty=y$. Then, we obtain
\begin{equation*}
  y=T y=T\left(T y \right)=T^{2} y=T^{2}\left( Ty\right)=T^{3} y=\cdots =T^{k}y.
\end{equation*}
It follows that $y$ is a fixed point of $T^{k}$, therefore $y=x_{0}$.
\end{proof}

The next theorem has been already stated in \cite{KhoMoMou}. Here, we present it with an alternative proof.

\begin{theorem}[Banach type Theorem]\label{PCBA2}
Let $(X,d)$ be a forward sequentially compact space and let $T: X\to X$ be a $b-$contraction. Assume that $f-$convergence implies $b-$convergence. Then, $T$ has a unique fixed point.
\end{theorem}
\begin{proof}
Let us take $x_{0}\in X$ and define the following sequence
\begin{equation*}
    x_{1}=Tx_{0},\, x_{2}=Tx_{1}=T^{2}x_{0},\, x_{3}=T^{3}x_{0},\,\ldots,\, x_{n}=T^{n}x_{0},\ldots .
\end{equation*}
Then, there exists $0<k<1$ such that
 \begin{align*}
    d(x_{m+1},x_{m})&=d(Tx_{m},Tx_{m-1})\\
    &\leq kd(x_{m-1},x_{m})=kd(Tx_{m-2},Tx_{m-1})\\
    &\leq k^{2}d\left(x_{m-1},x_{m-2}\right)=k^{2}d\left(Tx_{m-2},Tx_{m-3}\right)\\
    &\leq \cdots \leq\lambda k^{m},
 \end{align*}
with $\lambda=\max\left\lbrace d\left(x_{1},x_{0}\right), d\left(x_{0},x_{1}\right)\right\rbrace$. We will show that the sequence is $b-$Cauchy. Consider $m\leq n$, then
  \begin{align*}
    d(x_{n},x_{m}) &\leq d(x_{n},x_{n-1})+d(x_{n-1},x_{n-2})+\cdots+d(x_{m+1},x_{m})\\
    &\leq k^{n-1}\lambda+k^{n-2}\lambda+\ldots + k^{m}\lambda\\
    &=\left ( k^{m}+k^{m+1}+\ldots + k^{n-2}+k^{n-1}\right)\lambda=\left(\frac{k^{m}-k^{n}}{1-k}\right)\lambda\\
    &\leq \lambda\left(\frac{k^{m}}{1-k}\right).
 \end{align*}
 Since $0<k<1$, then $\displaystyle \frac{k^{m}}{1-k}\to 0$, when $m\to \infty$. Thus
 \begin{equation*}
    d(x_{n},x_{m})\leq \lambda\left(\frac{k^{m}}{1-k}\right)\to 0.
 \end{equation*}

Therefore, $\left\lbrace x_{n}\right\rbrace$ is $b-$Cauchy. By our assumption, there exists a subsequence $\left\lbrace x_{n_{k}}\right\rbrace$ of $\left\lbrace x_{n}\right\rbrace$ which is $f-$convergent, say $\displaystyle{x_{n_{k}} \stackrel{f}{\rightarrow} x}$. Then,  it follows from the Proposition~\ref{k12} that $d(x,x_{n})\to 0$.
  Let us see that $x$ is a fixed point of $T$.
 \begin{equation*}
    d(x,Tx)\leq d(x,x_{n})+d(Tx_{n-1},Tx)\leq d(x,x_{n})+kd(x,x_{n-1}).
 \end{equation*}
So, by letting $n\to \infty$, we have $d(x,Tx)=0$. By a similar argument, we obtain $d(Tx,x)=0$, which implies that $Tx=x$.

In order to prove uniqueness, suppose there exists $y\in X$ a fixed point of $T$ such that $y\neq x$. Then $d(x,y)\neq 0$ or $d(y,x)\neq 0$. Suppose $d(x,y)\neq 0$. Thus,
\begin{equation*}
     d(y,x)=d(Ty,Tx)< kd(x,y)=kd(Tx,Ty)< k^{2}d(y,x).
\end{equation*}
It follows that necessarily $d(y,x)\neq 0$, which implies that $1< k^{2}< 1$, leading to a contradiction.
\end{proof}

Now, we will present the asymmetric version of Edelstein's Theorem. In the symmetric case, this result is obtained by changing the hypothesis that $T$ is a contraction for $T$ a shrinkage mapping in Theorem~\ref{PCBA1}; this requires changing the completeness hypothesis to compactness.

\begin{theorem}[Edelstein type Theorem]\label{res1}
Let $(X,d)$ be an $f-$compact asymmetric space and let $T: X\to X$ be an $f-$shrinkage mapping. If $f-$convergence implies $b-$convergence, then $T$ has a unique fixed point.
\end{theorem}
\begin{proof}
Let $g\,:X\to\RR$ be given by $g(z)=d(z,Tz)$. We will prove that $g$ is $f-$continuous in $X$. Let $\lbrace x_{n}\rbrace$ be such that $x_{n} \stackrel{f}{\rightarrow} x$. Then, by inequality (\ref{fct}) we have that $Tx_{n} \stackrel{f}{\rightarrow} Tx$. Also $x_{n} \stackrel{b}{\rightarrow} x$ and $Tx_{n} \stackrel{b}{\rightarrow} Tx$. Let $\epsilon>0$, there exists $N\in \N$ such that if $n\geq N$,
\begin{equation*}
 d(x,x_{n})<\epsilon/2,\quad d(Tx,Tx_{n})<\epsilon/2, \quad d(x_{n},x)<\epsilon/2 \quad\text{and}\quad d(Tx_{n},Tx)<\epsilon/2.
\end{equation*}
Hence, by triangular inequality applied forward and backward,
\begin{equation*}
\left\vert g(x)-g(x_{n})\right\vert= \left\vert d\left(x, Tx\right)-d\left(x_{n}, T x_{n}\right)\right\vert<\epsilon,
\end{equation*}
therefore, $g$ is $f-$continuous.

Now Proposition~\ref{prop6} shows that $g$ takes its minimum value $m$ at $x_{0}\in X$. We show that $x_0$ is a fixed point of $T$. Suppose $x_{0}\neq Tx_{0}$. Then, inequality (\ref{fct}) implies
\begin{equation*}
    m=g(x_{0})\leq g(Tx_{0})=d(Tx_{0},T^{2}x_{0})<d(x_{0},Tx_{0})=m,
\end{equation*}
which is impossible. It remains to see that $x_{0}$ is the unique fixed point, which follows immediately.
\end{proof}

In the following result, we will see that if $T:K\to K$ is an $f-$non-expansive map defined on $K\subset X$, with $X$ an asymmetric space, then, under additional conditions on $K$ or $X$, we can guarantee the existence of fixed points of the mapping $T$. 

\begin{theorem}[Schauder type Theorem]\label{res2}
Let $X$ be an asymmetric normed space and let $K\subset X$ be an non-empty $f-$compact and convex subset. If $T:K\to K$ is $f-$non-expansive, and the $f-$convergence implies $b-$convergence, then $T$ has a fixed point.
\end{theorem}
\begin{proof}
Let $x_{0}\in K$. Let us define the following sequence of functions
    \begin{equation*}
    S_{n}x=\left(1-\frac{1}{n}\right)Tx+\frac{1}{n}x_{0},\quad\text{for each}\quad n\in\N,\quad x\in K.
    \end{equation*}
Then $S_{n}(K)\subseteq K$, for every $n\in\N$. Let  $C_{n}=\left(1-\frac{1}{n}\right)$, note that $0<C_{n}<1$, for all $n \in \mathbb{N}$, then
    \begin{equation*}
    \| S_{n}y-S_{n}x \mid =C_{n}\left\Vert Ty-Tx\right\vert \leq C_{n}\| y-x\mid.
    \end{equation*}
This implies that $S_{n}: K\to K$ is an $f-$contraction, for every $n\in \N$. Since $K$ is $f-$compact, Proposition~4.8 in \cite{ColZim} proves that $K$ is $f-$complete and by Theorem~\ref{PCBA1}, $S_{n}$ has a fixed point $x_{n}\in K$.
     On the other hand, taking into account that $K$ is $f-$compact and $T(K)\subset K$, there exists a subsequence $\lbrace Tx_{n_{j}}\rbrace$ and $u\in K$ such that $Tx_{n_{j }} \stackrel{f}{\rightarrow} u$, when $j\to \infty$. Even more,
\begin{equation*}
\left\Vert x_{n_{j}}-u\right\vert\leq \left(1-\frac{1}{n_{j}}\right) \left\Vert Tx_{n_{j}}-u\right\vert +\frac{1}{n_{j}}\left(\left\Vert  x_{0}\right\vert+\left\Vert -u\right\vert \right).
    \end{equation*}
    So $x_{n_{j}}\stackrel{f}{\rightarrow} u$, when $j\to\infty$. Therefore $Tx_{n_{j}} \stackrel{f}{\rightarrow} Tu$, when $j\to \infty$. Thus, $u$ is a fixed point of $T$.
\end{proof}

We have the following consequence.

\begin{theorem}\label{Teo3}
Let $X$ be an $f-$Banach space, and let $K$ be an $f-$closed, $f-$bounded, convex and non-empty subset of $X$. If $T:K \rightarrow K$ is $f-$non-expansive, $(T-I)(K)$ is an $f-$closed subset of $X$ and the $f-$convergence implies $b-$convergence, then $T$ has a fixed point at $K$.
\end{theorem}
\begin{proof}
Let $x_{0}\in K$ be given. Then, there exists $r>0$ such that $K\subseteq B^{f}(x_{0},r)$. Thus,
\begin{equation}
    \Vert Tx-Tx_{0}\vert\leq\Vert x-x_{0}\vert<r,\quad\text{for all} \quad x\in K.
\end{equation}
Consider the sequence $\left\lbrace t_{n}\right\rbrace$ given by $t_{n}=\frac{n}{n+1}$. Note that $0<t_{n}<1$ and $t_{n} \rightarrow 1$. For all $n\in \mathbb{N}$, let us define the mapping $T_{n}: K\rightarrow K$ by means of $T_{n}x=t_{n}Tx$. We will prove that $T_{n}$ is an $f-$contraction. Indeed, for all $x,y\in K$
\begin{equation*}
 d(T_{n}x,T_{n}y)=t_{n}\Vert Ty-Tx\vert \leq t_{n}\Vert y-x\vert=t_{n}d(x,y).
\end{equation*}
Then, from Theorem~\ref{PCBA1}, for every $n\in\N$ there exists a unique $x_{n}\in K$ such that $T_{n}x_{n}=x_{n}$.

On the other hand, let $\left\lbrace y_{n}\right\rbrace$ be the sequence in $(T-I)(K)$ given by $y_{n}=(T-I)(x_{n})$. We shall verify that $\displaystyle{y_{n} \stackrel{f}{\rightarrow} 0}$. In fact,
\begin{align*}
  d(0,y_{n})&=d(0,(Tx_{n}-x_{n}))=\left\Vert Tx_{n}-T_{n}x_{n}\right\vert\\
      &=\left\Vert Tx_{n}-t_{n}Tx_{n}\right\vert=\left(1-t_{n}\right)\left\Vert T x_{n}\right\vert\\
      &\leq\left(1-t_{n}\right)\left(\left\Vert Tx_{n}-Tx_{0}\right\vert + \left\Vert Tx_{0}\right\vert\right)\\
      &\leq 2\left(1-t_{n}\right)\max\left\lbrace r, \left\Vert Tx_{0}\right\vert\right\rbrace.
\end{align*}
    Since $t_{n} \rightarrow 1$, then $\displaystyle{y_{n} \stackrel{f}{\rightarrow} 0}$. Therefore, $0\in (T-I)(K)$, which implies that there exists $x_{0}\in K$ such that $(I-T)(x_{0})=0$. Hence, $0=(T-I)(x_{0})=Tx_{0}-x_{0}$. Thus, $T$ has a fixed point at $K$, as desired.
 \end{proof}

The following result is a consequence of Theorem~\ref{PCBA1} and it will be a fundamental tool in the proof of the Theorem~\ref{Kirk}.
\begin{lemma}\label{SAPF}
Let $X$ be an $f-$Banach space, and let $K\subseteq X$ be a non-empty, convex, $f-$closed and $f-$bounded subset. Consider $T: K\to K$ an $f-$non-expansive mapping. Suppose that $f-$convergence implies $b-$convergence, then there exists a sequence $\lbrace x_{n}\rbrace\subset K$
such that
\begin{equation}\label{approx}
\lim_{n\to \infty}\Vert x_{n}-Tx_{n}\vert = 0.
\end{equation}
This sequence will be called a fixed-point $f-$approximating sequence for $T$ in $K$.
\end{lemma}
\begin{proof}
Let $a_{0}\in K$. Define the mapping $S_{n}:K\to K$, given by
\begin{equation*}
S_{n}x=\frac{a_{0}}{n}+\left(1-\frac{1}{n}\right)Tx,\quad\text{if}\quad n\geq 2.
\end{equation*}
It is easy to see that $S_{n}$ is an $f-$contraction for every $n\geq 2$. Then, by Theorem~\ref{PCBA1}, we can state that $S_{n}$ has a unique fixed point $x_{n}$ for each $n\geq 2$.

Now, because $S_{n}x_{n}=x_{n}$, we can write
\begin{equation*}
x_{n}=\frac{a_{0}}{n}+\left(1-\frac{1}{n}\right)Tx_{n},\quad\text{if}\quad n\geq 2.
\end{equation*}

We know that there exists $b\in K$ and $M>0$ such that $d(b,x)=\Vert x-b\vert<M$, for all $x\in K$, because $K$ is $f-$bounded. Also,
\begin{align*}
\Vert Tx_{n}-x_{n}\vert=\frac{1}{n}\left\Vert Tx_{n}-a_{0}\right\vert.
\end{align*}
Since $K$ is $T-$invariant, we can choose $a_{0}$ in such a way that $a_{0}=Tb$. Then, taking into account that $T$ is $f-$non-expansive, we obtain
\begin{align*}
\Vert Tx_{n}-x_{n}\vert &=\frac{1}{n}\left\Vert Tx_{n}-a_{0}\right\vert=\frac{1}{n}\left\Vert Tx_{n}-Tb\right\vert \leq \frac{1}{n}\left\Vert x_{n}-b\right\vert\leq \left(\frac{1}{n}\right)M.
\end{align*}
Thus, $\Vert Tx_{n}-x_{n}\vert\to 0$ when $n\to \infty$. Now, considering the sequence \mbox{$\{ z_{n}=Tx_{n}-x_{n}\}\subseteq X$}, and the fact that $f-$convergence implies $b-$convergence, we obtain (\ref{approx}).
\end{proof}

Next, we present an asymmetric version of the Goebel-Karlovitz Lemma, which together with Lemma~\ref{SAPF}, constitute the central part of the proof of Theorem~\ref{Kirk}.

\begin{lemma}[Goebel-Karlovitz type Lemma]\label{GK}
Let $X$ be an $f-$Banach space, $K\subset X$ non-empty, convex, $f-$closed, weakly $f-$compact and $f-$bounded. Consider $T:K\to K$ an $f-$non-expansive map. Suppose $f-$convergence implies $b-$convergence. If $K$ is minimal $T-$invariant, then
    \begin{equation*}
        \lim_{n\to \infty} \Vert x_{n}-x\vert=\Diam (K),\quad\text{for all}\quad x\in K,
    \end{equation*}
  where $\lbrace x_{n}\rbrace$ is a fixed-point $f-$approximating sequence for $T$ in $K$.
\end{lemma}
\begin{proof}
Since $K$ is $f-$bounded, there exists $y_{0}\in K$, $M>0$ such that $\Vert x-y_{0}\vert <M$, for all $x\in K$. By Lemma~\ref{SAPF}, there exists $\left\lbrace x_{n}\right\rbrace\subseteq K$ such that
\begin{equation}
\lim_{n\to \infty}\Vert x_{n}-Tx_{n}\vert = 0.
\end{equation}
Consider $\displaystyle s_{0}=\limsup_{n\to \infty}\Vert x_{n}-y_{0}\vert\geq 0$ and
\begin{align}\label{D}
D=\left\lbrace x\in K\,:\,  \limsup_{n\to \infty}\Vert x_{n}-x\vert\leq s_{0} \right\rbrace.
\end{align}
Note that $D\neq\emptyset$, since $y_{0}\in D$. It is easy to see that $D$ is convex, $f-$closed, and $T-$invariant. Then, by the minimality of $K$, we have $D=K$. The same argument is valid if $y_{0}$ is changed by another $y\in K$, thus $D_{y}=K$ whenever $y\in K$.

For $y\in K$, consider the set
\begin{equation}\label{P(y)}
P(y)=\left\lbrace \Vert x_{n}-y\vert\,\vert\,n\in\N  \right\rbrace\subset \RR,
\end{equation}
and denote by $P(y)'$ the set of limit points of $P(y)$. Then, considering $s'\in P(y)'$, there exists $\left\lbrace \beta_{m}(y)\right\rbrace\subset P(y)$ such that $\displaystyle{\beta_ {m}(y)\rightarrow s'}$, that is,
\begin{equation*}
\beta_{m}(y)=\Vert x_{n_{m}}-y\vert\rightarrow s'.
\end{equation*}
Let us prove that $\displaystyle{ \beta_{m}(z)\rightarrow s'}$ for all $z\in K$. Assume the opposite, that is, there exists $z_{0}\in K$ such that $\displaystyle{ \beta_{m}(z_{0})\nrightarrow s'}$. Then, there exists $\left\lbrace m_{j}\right\rbrace$ such that $\displaystyle{\beta_{m_{j}}(z_{0})\rightarrow t}$, with $t\neq s' $. Now, introduce the set
\begin{equation}\label{E}
E=\left\lbrace w\in K\,:\,\limsup_{j\to \infty}\Vert x_{m_{j}}-w\vert\leq \min\left\lbrace t,s'\right\rbrace\right\rbrace,
\end{equation}
note that $z_{0}\in E$, so that $E\neq \emptyset$. With similar arguments we can prove that the set $E$ is also $f-$closed, convex and $T-$invariant. But, the minimality of $K$, implies that $E=K$. This means that $y,z_{0}\in E$.

On the other hand, if $t\neq s'$, then $s'<t$ or $t<s'$. Without loss of generality, suppose that $t<s'$. Since $y\in E$, we have that
\begin{align*}
s'=\limsup_{j\to \infty}\Vert x_{m_{j}}-y\vert\leq \min\left\lbrace t,s'\right\rbrace =t<s',
\end{align*}
which is a contradiction. Therefore, $t=s'$ and $\displaystyle \lim_{m\to \infty}\Vert x_{n_{m}}-z\vert=s'$, for all $z\in K $.

Next, we show that $s'=d$ where $d=\Diam(K)$. Introduce the set
\begin{equation}\label{F}
F=\left\lbrace u\in K\,:\,\Vert u-x\vert\leq s';\quad\text{for each}\quad x\in K\right\rbrace,
\end{equation}
since $K$ is weakly $f-$compact, then we can extract a subsequence $\lbrace x_{n_{j}}\rbrace$, which is weakly $f-$convergent, say $\displaystyle{x_{n_{j}}\stackrel{f}\rightharpoonup w_{0}}$, with $w_{0}\in K$. Then, Theorem~\ref{Mazur} guarantees that $F\neq \emptyset$.

It is not difficult to verify that $F$ is convex and $f-$closed. To show that $F$ is $T-$invariant, consider  $v_{1}\in F$. By Lemma~\ref{mch}, we have that $K=\overline{conv}(T(K))$. Let $v\in K$, and $\epsilon>0$, we choose $\displaystyle v_{0}=\sum_{i=1}^{m}\lambda_{i}Ty_{i}$, with $y_{i}\in K$, $\lambda_{i}>0$, $\displaystyle\sum_{i=1}^{m}\lambda_{i}=1$ and $\Vert v_{0}-v\vert <\epsilon$.
\begin{align*}
\Vert Tv_{1}-v\vert &=\Vert Tv_{1}-v+v_{0}-v_{0}\vert\\
&\leq \Vert Tv_{1}-v_{0}\vert + \Vert v_{0}-v\vert\\
&= \left\Vert Tv_{1}-\sum_{i=1}^{m}\lambda_{i}Ty_{i}\right\vert + \Vert v_{0}-v\vert\\
&\leq s' + \Vert v_{0}-v\vert < s'+\epsilon,
\end{align*}
hence $Tv_{1}\in F$, which shows that $F$ is $T-$invariant.

We have shown that $F$ is a non-empty, $f-$closed, convex, and $T-$invariant subset of $K$. From the minimality of $K$, there follows that $F=K$. Now, suppose that $s'<d=\Diam(K)$. Define $\displaystyle \delta=\frac{s'+d}{2}$, then $s'<\delta<d$. Because $F=K$, we have
\begin{align*}
    \Vert v-x\vert\leq s'<\delta,\quad \text{for all}\quad v,x\in K.
\end{align*}
So,
\begin{align*}
d=\Diam(K)=\sup_{v,x\in K}\{\Vert v-x\vert\}\leq \delta<d,
\end{align*}
which is a contradiction. Therefore, $s'=d=\Diam(K)$.

Finally, taking into account that the set $P(y)\subseteq\RR$ defined in (\ref{P(y)}) is bounded for all $y\in K$, it means that $P(y)'=\lbrace d\rbrace$ for all $y\in K$. It implies that
\begin{equation*}
    \lim_{n\to \infty}\Vert x_{n}-y\vert =s'=\Diam(K), \quad\text{for every}\quad y\in K.  
\end{equation*}
\end{proof}

\begin{theorem}[Kirk type theorem]\label{Kirk}
Let $X$ be an $f-$Banach space, and let \mbox{$K\subset X$} be non-empty, convex, $f-$closed, $f-$bounded and weakly $f-$compact. Consider \mbox{$T:K\to K$} an $f-$non-expansive mapping. Suppose that $K$ has an $f-$normal structure and that $f-$convergence implies $b-$convergence in $X$, then $T$ has a fixed point in $K$.
\end{theorem}
\begin{proof}
Since $K$ is weakly $f-$compact, by Proposition~\ref{minimal} there follows that $K$ has a subset $D$ which is minimal $T-$invariant. By Lemma~\ref{SAPF}, there exists an $f-$approximating  sequence $\lbrace x_{n}\rbrace$ that is $f-$convergent to the fixed point of $T$ in $D$, such that
\begin{equation*}
\lim_{n\to\infty}\Vert x_{n}-Tx_{n}\vert=0.
\end{equation*}
Then, by Lemma~\ref{GK}, $\displaystyle\lim_{n\to\infty}\Vert x_{n}-x\vert=\Diam(D)$, for all $x\in D$. 

We will see that $D$ is a single element set. Suppose the opposite, that is, $D$ consists of more than one point. Because $D$ is an $f-$closed, convex, non-empty subset of $K$, and since $K$ has an $f-$normal structure, then there exists a forward non-diametral point $x_{0}\in D$, which means that $\sup_{n\in \N}\Vert x_{n}-x_{0}\vert<\Diam (D)$.  Thus,
\begin{equation*}
\lim_{n\to\infty}\Vert x_{n}-x_{0}\vert \leq \sup_{n\in\N}\Vert x_{n}-x_{0}\vert<\Diam (D),
\end{equation*}
but it contradicts Lemma~\ref{GK}. Therefore, $D$ is a single-element set. Now, since $D$ is $T-$invariant, that is, $T(D)\subseteq D$, then this point is necessarily a fixed point of $T$.
\end{proof}
It is important to mention that, in classical theory, there are other fixed point results that use very interesting ideas and techniques that could be extended to the context of asymmetric spaces.  For example, the measure of non-compactness and Darbo's fixed point theorem are not treated in the present work, whose classic versions can be consulted in \cite{AyDoLo}. Also, there is the notion of exterior spaces, which are examples of asymmetric spaces, see \cite{AlKfBest}, on which the possibility of establishing fixed-point results could be considered.


\section*{Acknowledgments}
The first author acknowledges support from project FAB-06-19 and the second author acknowledges support from Conacyt project 45886.

\section*{References}


\end{document}